\documentclass[11pt,]{article}
\usepackage{color}
\usepackage{amssymb,amsmath}
\usepackage{amsthm}
\usepackage{hyperref}
\usepackage{mathrsfs}
\usepackage{textcomp}
\usepackage{verbatim}
\usepackage{marginnote}

\usepackage{sectsty}

\hypersetup{
    colorlinks,%
    citecolor=black,%
    filecolor=black,%
    linkcolor=black,%
    urlcolor=black
}

\makeatletter

\newdimen\bibspace
\renewenvironment{thebibliography}[1]{%
 \section*{\refname %or \bibname if you use ``book'' as the documentclass
       \@mkboth{\MakeUppercase\refname}{\MakeUppercase\refname}}%
     \list{\@biblabel{\@arabic\c@enumiv}}%
          {\settowidth\labelwidth{\@biblabel{#1}}%
           \leftmargin\labelwidth
           \advance\leftmargin\labelsep
           \itemsep\bibspace
           \parsep\z@skip     %
           \@openbib@code
           \usecounter{enumiv}%
           \let\p@enumiv\@empty
           \renewcommand\theenumiv{\@arabic\c@enumiv}}%
     \sloppy\clubpenalty4000\widowpenalty4000%
     \sfcode`\.\@m}
    {\def\@noitemerr
      {\@latex@warning{Empty `thebibliography' environment}}%
     \endlist}

\makeatother

\makeatletter
\newtheorem{trm}{Theorem}
\newtheorem*{trma}{Theorem A}
\newtheorem*{trmb}{Theorem B}
\newtheorem*{trmc}{Theorem C}
\newtheorem{prop}[trm]{Proposition}
\newtheorem{cor}[trm]{Corollary}
\newtheorem{rem}[trm]{Remark}
\newtheorem{lemma}[trm]{Lemma}
\newtheorem{defin}[trm]{Definition}

\def\Xint#1{\mathchoice
  {\XXint\displaystyle\textstyle{#1}}%
  {\XXint\textstyle\scriptstyle{#1}}%
  {\XXint\scriptstyle\scriptscriptstyle{#1}}%
  {\XXint\scriptscriptstyle\scriptscriptstyle{#1}}%
  \!\int}
\def\XXint#1#2#3{{\setbox0=\hbox{$#1{#2#3}{\int}$}
  \vcenter{\hbox{$#2#3$}}\kern-.5\wd0}}

\def\dashint{\Xint-}

\newcommand{\R}[1]{\mathbb{R}^{#1}}

\newcommand{\de}{\partial}
\newcommand{\ve}{\varepsilon}
\newcommand{\M}[1]{\mathcal{#1}}

\newcommand{\ud}{\mathrm{d}}

\DeclareMathOperator{\loc}{loc}

\title{\textbf{Existence and asymptotics for solutions of a non-local Q-curvature equation in dimension three}}
\author{Tianling Jin \and  Ali Maalaoui \and Luca Martinazzi \and  Jingang Xiong}

\begin{document}
\maketitle

\begin{abstract}
We study conformal metrics on $\R{3}$, i.e., metrics of the form $g_u=e^{2u}|dx|^2$, which have constant $Q$-curvature and finite volume. This is equivalent to studying the non-local equation
\begin{equation*}%\label{eq00}
(-\Delta)^\frac32 u = 2 e^{3u}\text{ in }\R{3},\quad V:=\int_{\R{3}}e^{3u}dx<\infty,
\end{equation*}
where $V$ is the volume of $g_u$. Adapting a technique of A. Chang and W-X. Chen to the non-local framework, we show the existence of a large class of such metrics, particularly for $V\le 2\pi^2=|\mathbb{S} ^3|$. Inspired by previous works of C-S. Lin and L. Martinazzi, who treated the analogue cases in even dimensions, we classify such metrics based on their behavior at infinity.
\end{abstract}

\section{Introduction}
In this paper, we study existence and asymptotics for solutions of
\begin{equation}\label{eq0}
(-\Delta)^\frac{3}{2} u =2e^{3u}\quad \text{in }\R{3}
\end{equation}
with
\begin{equation}\label{area}
V:=\int_{\R{3}}e^{3u}\,\ud x<\infty,
\end{equation}
where $(-\Delta)^\frac{3}{2}$ is interpreted as $(-\Delta)^\frac{1}{2}\circ(-\Delta)$. To define $(-\Delta)^\frac{1}{2} v$ for a function $v$ in $\R{3}$, we require throughout the paper that 
$$v\in L_{1/2}(\R{3}):=\left\{v\in L^1_{\loc}(\R{3}):\int_{\R{3}}\frac{|v(x)|}{1+|x|^4}dx<\infty \right\},$$
which makes $(-\Delta)^\frac{1}{2} v$ be a tempered distribution (see \cite{Sil}).

\begin{defin}\label{deflapl}
Given a tempered distribution $f$ in $\R{3}$, we say that $u$ is a solution of $(-\Delta)^{\frac32} u=f$ if $u\in W^{2,1}_{loc}(\R{3})$, $\Delta u\in L_{1/2}(\R{3})$ and
\begin{equation}\label{deffrac}
\int_{\R{3}} (-\Delta u)\,  (-\Delta)^\frac12\varphi  dx=\langle f ,\varphi\rangle\quad \text{for every }\varphi\in\mathcal{S}(\R{3}),
\end{equation}
where $\mathcal{S}(\R{3})$ is the Schwarz space of rapidly decreasing smooth functions in $\R{3}$.
\end{defin}
%Note that if $u\in W^{2,1}_{loc}(\R{3})$, $\Delta u\in L_{1/2}(\R{3})$ then
%\[
%\langle (-\Delta)^\frac12(-\Delta u), \varphi\rangle= \int_{\R{3}} (-\Delta u)\,  (-\Delta)^\frac12\varphi  dx.
%\]
Note that the LHS of \eqref{deffrac} is finite since $\Delta u\in L_{1/2}(\R{3})$ and \eqref{stimaphi} holds.

Equation \eqref{eq0} is a prescribed $Q$-curvature equation, in the sense that if a smooth function $u$ solves
$$(-\Delta)^\frac{3}{2}u=Ke^{3u}\quad\text{in }\R{3}$$
for some function $K$, then the metric $g_u:=e^{2u}|dx|^2$ (which is a conformal perturbation of the Euclidean metric $|dx|^2$) has $Q$-curvature $K$, see e.g. \cite{cha}, \cite{CY2} or \cite{GZ} and the references therein. Moreover, the quantity $V$ appearing in \eqref{area} is simply the volume of $g_u$.

Problem \eqref{eq0}-\eqref{area} is the three dimensional case of the problem
\begin{equation}\label{eqn}
(-\Delta)^\frac{n}{2}u=(n-1)! e^{nu}\text{ in }\R{n},\quad V:=\int_{\R{n}}e^{nu}dx<\infty,
\end{equation}
which has been received considerable attentions, particularly in the case $n$ even.
It is well-known that the function $w_0(x):=\log\left(\frac{2}{1+|x|^2}\right)$ is a solution of \eqref{eqn} with $V=|\mathbb{S}^n|$ for any $n\ge 1$. Indeed, $w_0$ has the following geometric interpretation: If $\pi: \mathbb{S}^n\setminus \{p\}\to\R{n}$ is the stereographic projection from the sphere $\mathbb{S}^n=\{x\in \R{n+1}:|x|=1\}$ minus the south pole $p$ given by
$$\pi(x',x_{n+1})=\frac{x'}{1+x_{n+1}},\quad x'=(x_1,\dots, x_n)$$
and $g_0$ is the round metric on $\mathbb{S}^n$, then
$$(\pi^{-1})^* g_0= e^{2w_0}|dx|^2.$$
Applying the M\"obius transformations (translations and dilations) to $w_0$ (or to $(\pi^{-1})^* g_0$ to be more precise), we obtain the functions
\begin{equation}\label{wspherical}
w_{x_0,\lambda}(x)=\log\left(\frac{2\lambda}{1+\lambda^2|x-x_0|^2}\right),\quad x_0\in\R{n},\quad \lambda>0,
\end{equation}
which also solve \eqref{eqn} with $V=|\mathbb{S}^n|$. Because of their geometric origin, they can be called \emph{spherical} solutions.

In dimension $2$, where \eqref{eqn} reduces to $-\Delta u=e^{2u}$, it was proven by Chen-Li \cite{CL} that all solutions of \eqref{eqn} are spherical. Things are different in higher dimensions as shown by A. Chang and W-X. Chen \cite{CC}.

\begin{trma}[\cite{CC}]\label{thm:thma}
For every $n\ge 4$ \emph{even} and $V\in (0,|\mathbb{S}^{n}|)$
there exists a (non-spherical) solution $u\in C^\infty(\R{n})$ of \eqref{eqn}.
\end{trma}

The restriction to $n$ even in Theorem A is essentially technical: for $n$ odd the operator $(-\Delta)^\frac{n}{2}$ is non-local and several difficulties arise. On the other hand, we will show, at least in dimension $3$ that the arguments in \cite{CC} can be adapted to the non-local setting.

\begin{trm}\label{nonstandard}
For every $V\in (0,2\pi^2)$, Problem \eqref{eq0}-\eqref{area} has at least one solution (in the sense of Definition \ref{deflapl}) $u\in C^\infty(\R{3})$.
\end{trm}

It is natural to try to gather information about the non-spherical smooth solutions produced by Theorem \ref{nonstandard}, in particular their behavior at infinity.
To do that, let us first recall that the fundamental solution of $(-\Delta)^\frac{3}{2}$ in $\R{3}$ is 
$$\Gamma(x):=\frac{1}{2\pi^2}\log\left(\frac{1}{|x|}\right),$$
i.e.
$(-\Delta)^\frac{3}{2}\Gamma=\delta_0$  in $\R{3}$ in the sense of tempered distributions. This follows, e.g., from $\Delta \log|x|=|x|^{-2}$ and Lemma \ref{lemmafund} below. Set
\begin{equation}\label{eqalpha}
\alpha:=\frac{1}{\pi^2}\int_{\R{3}}e^{3u}\,dx
\end{equation}
and
\begin{equation}\label{eqv}
v(x):=\frac{1}{\pi^2}\int_{\R{3}}\log\left(\frac{|y|}{|x-y|}\right) e^{3u(y)}dy,
\end{equation}
where $u$ is a smooth solution of \eqref{eq0}-\eqref{area}.
The function $v$ looks quite similar to $\Gamma * 2e^{3u}=\Gamma* (-\Delta)^\frac{3}{2}u$ (except for the additional $|y|$ appearing in the argument of the logarithm, which is necessary to make the integral in \eqref{eqv} convergent, but which plays no role after one differentiates $v$). In fact, as we shall see in Lemma \ref{equal} that $(-\Delta)^\frac{3}{2}v=(-\Delta)^\frac{3}{2}u$, it is reasonable to ask how $u$ and $v$ are related. Since for any polynomial $p$ of degree at most $2$ one has $(-\Delta)^\frac{3}{2}p=0$, one could wonder whether $u-v=p$ for a polynomial of degree $0$ (a constant), $1$ or $2$. It turns out that this is the case, and $p$ is either a constant or a polynomial of degree $2$ bounded from above. Moreover, $v$ exhibits a well-controlled behavior at infinity.

\begin{trm}\label{clas1} Let $u$ be a smooth solution of \eqref{eq0} satisfying \eqref{area}. Then
\begin{equation}\label{repr}
u=v+p,
\end{equation}
where $p$ is a polynomial of degree $0$ or $2$ bounded from above, $v$ is as in \eqref{eqv}. Moreover, $v$ satisfies
\begin{equation}\label{limDeltav}
\lim_{|x|\to\infty}\nabla^\ell v(x)=0,\quad \text{for }\ell =1,2,
\end{equation}
\begin{equation}\label{limv}
v(x)=-\alpha\log|x|+o(\log|x|),\quad \text{as }|x|\to\infty,
\end{equation}
where $\alpha>0$ is given by \eqref{eqalpha}.
\end{trm}

The behavior at infinity of $u$ in terms of the decomposition $u=v+p$ in Theorem \ref{clas1}  can be used to give necessary and sufficient conditions under which a solution of \eqref{eq0}-\eqref{area} is spherical. This is the content of the following theorem.

\begin{trm}\label{clas2} Let $u$ be a smooth solution of \eqref{eq0} satisfying \eqref{area}. Then the following are equivalent:
\begin{itemize}
\item[(i)] $u$ is a spherical solution, i.e. $u=w_{x_0,\lambda}$ as in \eqref{wspherical} for some $\lambda>0$, $x_0\in\R{3}$;
\item[(ii)] $\deg p=0$, where $p$ is the polynomial in \eqref{repr};
\item[(iii)] $\lim_{|x|\to\infty}\Delta u(x)=0$;
\item[(iv)] $u(x)=o(|x|^2)$ as $|x|\to \infty$;
\item[(v)] $\liminf_{|x|\to +\infty}R_{g_u}> -\infty$, where $g_u=e^{2u}|dx|^2$, and $R_{g_u}$ is the scalar curvature of $g_u$;
\item[(vi)] $\pi^*g_u$ can be extended to a Riemannian metric on $\mathbb{S}^{3}$, where $\pi:\mathbb{S}^{3}\setminus\{p\}\to\R{3}$ is the 
stereographic projection and $p\in \mathbb{S}^3$ is the south pole.
\end{itemize}
Moreover, if $u$ is not a spherical solution then there exists a constant $a>0$ such that
\begin{equation}\label{deltaa}
\Delta u(x)\to -a\quad \textrm{as }|x|\to+\infty.
\end{equation}
\end{trm}

Conclusions similar to those of Theorem \ref{clas1} for solutions of \eqref{eqn} were proven by C-S. Lin \cite{lin1} in dimension $4$ and by L. Martinazzi \cite{mar1} in arbitrary \emph{even} dimension. Also Theorem \ref{clas2}, in this generality, was proven by Lin in dimension 4 and Martinazzi in arbitrary even dimension, extending several previous results in \cite{CY2, WX1, Xu}.

It is also interesting to investigate what values the volume $V$ in \eqref{area} can be attained. According to Theorem \ref{nonstandard}, and in analogy with Theorem A, every value in $(0,2\pi^2)$ can be attained, and of course the value $V=2\pi^2$ is attained by the spherical solutions. Can $V$ attain values bigger than the volume of $\mathbb{S}^3$? The corresponding question in dimension $4$ was answered in the negative by C-S. Lin \cite{lin1}, which shows that Theorem A is sharp as far as $V$ is concerned.

\begin{trmb}[\cite{lin1}]
For every non-spherical solution of Problem \eqref{eqn} with $n=4$ one has $V<|\mathbb{S}^4|$.
\end{trmb}

 Surprisingly, it was recently shown by Martinazzi \cite{mar2} that in dimension $n=6$ things are quite different and \eqref{eqn} has solutions for $V$ arbitrarily large.

\begin{trmc}[\cite{mar2}]
There exist $V_*>|\mathbb{S}^6|$ and $V^*>0$ such that for every $V\in (0,V_*]$ and for every $V\ge V^*$ there exists a solution of \eqref{eqn} with $n=6$.
\end{trmc}

It turns out that in dimension $3$ Problem \eqref{eqn} behaves like in dimension $4$ and not like in dimension 6. More precisely:

\begin{trm}\label{trmvolume} Let $u$ be a non-spherical smooth solution of \eqref{eq0}-\eqref{area}. Then $V<2\pi^2$.
\end{trm}

Let us spend a few words about the potential applications of Theorems \ref{clas1} and \ref{clas2}. In even dimension $n=2m$, their analogs (compare to Problem \eqref{eqn}) have been widely used to prove compactness, quantization and existence results for equations of order $2m$ with critical growth, such as the equation
\begin{equation}\label{eqMT}
(-\Delta)^m u=\lambda u e^{mu^2},\quad \lambda>0
\end{equation}
satisfied by critical points of some Moser-Trudinger type inequality, see, e.g., \cite{AD, AS, dru, LRS, mar4, MS2, RS, str},
or the equation
\begin{equation}\label{eqQcurv}
P^{2m}_g u+ Q_g= Q e^{2mu}\quad \text{on a manifold } (M^{2m},g)
\end{equation}
which prescribes the $Q$-curvature of the manifold $(M, e^{2u}g)$, see, e.g., \cite{DR, mal, MS, mar3, str0},
or to the higher order Liouville equation 
\begin{equation}\label{eqopen}
(-\Delta)^m u= Ve^{2mu}\quad \text{in }\Omega\subset\R{2m},\quad V\in L^\infty(\Omega),
\end{equation}
see, e.g., \cite{mar5, rob1, rob2, RW, MP}. The main idea is that if a sequence $\{u_k\}$ of solutions (or the heat flow) of \eqref{eqMT}, \eqref{eqQcurv} or \eqref{eqopen} is not pre-compact, then a suitably blown-up  subsequence will converge strongly (say in $C^{2m}_{\loc}(\R{2m})$) to a solution of \eqref{eqn}. Then it is understandably important to know the behavior of the solutions of \eqref{eqn}, and in particular to have geometric or analytic conditions which ensure that a solution is spherical. Therefore, we expect that the above Theorems \ref{clas1} and \ref{clas2} will be useful in understanding the non-local analogs of  \eqref{eqMT}, \eqref{eqQcurv} and \eqref{eqopen} in dimension $3$.

\medskip

The paper is organized as follows. In Section \ref{GJMS} we start with some definitions and results which will be necessary to give a simple and essentially self-contained (up to Beckner's inequalities and the Sobolev embeddings) proof of Theorem \ref{nonstandard}. The proof of Theorem \ref{nonstandard} will be then given in Section \ref{s:nonex}, and it will follow from  Theorem \ref{trmexist}. In Section \ref{s:estimates} we prove the main lemmas which will be used to prove Theorems \ref{clas1}, \ref{clas2} and \ref{trmvolume}. In the appendix we collect a few definitions and theorems about the fractional Laplacian. 

\bigskip

\noindent\textbf{Acknowledgements:} We would like to thank anonymous referee for his/her suggestions and comments. A. Maalaoui and L. Martinazzi were supported in part by the Swiss National Science Foundation. J. Xiong was supported in part by the First Class Postdoctoral
Science Foundation of China (No. 2012M520002). Part of the work was done while A. Maalaoui and L. Martinazzi were attending the research program
``Conformal Geometry and Geometric PDE's'' at CRM Barcelona.

\section{Preliminaries}\label{GJMS}

Let $g_0$ be the standard metric on $\mathbb{S}^3$ and $\Delta_{g_{0}}$ be the Laplace-Beltrami operator. Let $\{\lambda_k=k(k+2), k\in\mathbb{N}\cup\{0\}\}$ be the eigenvalues of $-\Delta_{g_0}$. The eigenspace of $\lambda_k$ is of finite dimension $N_k$ and is spanned by spherical harmonics $Y^\ell_k $ of degree $k$, where $\ell=1,\cdots, N_k$  (see, e.g., \cite{Stein}). We renormalize them so that $\|Y^{\ell}_{k}\|_{L^2(\mathbb{S}^3)}=1$. The spherical harmonics $\{Y^{\ell}_{k}\}$ form an orthonormal basis of the Hilbert space $L^2(\mathbb{S}^3)$. In particular, given $u\in L^2(\mathbb{S}^3)$ we can write
\begin{equation}\label{uspher}
u = \sum_{k=0}^{\infty}\sum_{\ell=1}^{N_k} u_{k}^{\ell}Y^{\ell}_{k},\quad u_k^\ell\in\R{},
\end{equation}
and $\|u\|_{L^2(\mathbb{S}^3)}=\sum_{k,\ell} (u_k^\ell)^2$.

One can define an operator $P_{g_{0}}^{3}$ as follows (see e.g. \cite{Beck} and \cite{CQ}). Given $u \in L^{2}(\mathbb{S}^3)$ with spherical harmonics expansion as in \eqref{uspher} such that
\begin{equation}\label{def}
\|u\|^2_{H^3(\mathbb{S}^3)}:= \|u\|_{L^2}^2+\sum_{k=1}^{\infty}\sum_{\ell=1}^{N_k}|u^{\ell}_{k}|^2(\lambda_{k}+1)\lambda_{k}^2<\infty,
\end{equation} 
we define
$$P_{g_{0}}^{3}u:=\sum_{k=0}^{\infty}(\lambda_{k}+1)^{\frac{1}{2}}\lambda_{k}\sum_{\ell=1}^{N_k}u^{\ell}_{k}Y^{\ell}_{k}.$$
Notice that on $H^3(\mathbb{S}^3)$ the operator $P^{3}_{g_{0}}$ coincides with the operator $(-\Delta_{g_{0}}+1)^{\frac{1}{2}}(-\Delta_{g_{0}})$, where the operator $(-\Delta_{g_{0}}+1)^{\frac{1}{2}}$ is also understood in terms of spectral decomposition of the Laplace-Beltrami operator:
$$(-\Delta_{g_{0}}+1)^{\frac{1}{2}} u=\sum_{k=0}^\infty \sqrt{\lambda_k+1}\sum_{\ell=1}^{N_k} u^{\ell}_k Y^{\ell}_k,\quad \text{for $u$ as in \eqref{uspher}}.$$
Therefore, $P^3_{g_0}$ is the well-known intertwining operator on $\mathbb{S}^3$ (see, e.g., \cite{Bra}).

Define the space
$$H^{\frac{3}{2}}(\mathbb{S}^{3})=\left\{ u=\sum_{k=0}^\infty \sum_{\ell=1}^{N_k} u_k^\ell Y_k^\ell \in L^{2}(\mathbb{S}^3): \sum_{k=0}^{\infty}(\lambda_{k}+1)^{\frac{1}{2}}\lambda_{k}\sum_{\ell=1}^{N_k}|u^{\ell}_{k}|^{2}<\infty \right\},$$
endowed with the seminorm 
$$\|u\|_{\dot H^{3/2}}^2:= \sum_{k=0}^{\infty}(\lambda_{k}+1)^{\frac{1}{2}}\lambda_{k}\sum_{\ell=1}^{N_k}|u^{\ell}_{k}|^{2}$$
and with the norm
$$\|u \|_{H^{3/2}}^2:=\|u\|_{L^2}^2+\|u\|^2_{\dot H^{3/2}}.$$
%so that $H^{\frac{3}{2}}(\mathbb{S}^3)=\{u\in L^1(\mathbb{S}^3): \|u\|_{H^{3/2}}<\infty\}$.

Since the operator $P^{3}_{g_{0}}$ is self-adjoint and non-negative, one can define the operator $(P^{3}_{g_{0}})^{\frac{1}{2}}$ for functions $u\in H^\frac{3}{2}(\mathbb{S}^3)$ as
$$(P^{3}_{g_{0}})^{\frac{1}{2}}u=\sum_{k=1}^{\infty}(\lambda_{k}\sqrt{1+\lambda_{k}})^{\frac{1}{2}}\sum_{\ell=1}^{N_k}u^{\ell}_{k}Y^{\ell}_{k}.$$
Notice that $\|(P^{3}_{g_{0}})^{\frac{1}{2}}u\|_{L^2}=\|u\|_{\dot H^{3/2}}$. 

The (fractional) Sobolev embeddings give $H^\frac{3}{2}(\mathbb{S}^3)\hookrightarrow L^p(\mathbb{S}^3)$ for every $p\in [1,\infty)$ but not for $p=\infty$, in which case the following inequality is a useful replacement.

\begin{trm}[Theorem 1 in \cite{Beck}]\label{beck}
For every $u\in H^{\frac{3}{2}}(\mathbb{S}^3)$ one has
\begin{equation}\label{eq:beck}
\log \left(\dashint_{\mathbb{S}^3} e^{u-\overline{u}} dV_{0}\right)\leq \frac{1}{24\pi^2}\|u\|_{\dot H^{3/2}}^2,
\end{equation}
where $\dashint_{\mathbb{S}^3}=\frac{1}{|\mathbb{S}^3|}\int_{\mathbb{S}^3}$, $\overline u$ is the average of $u$ on $\mathbb{S}^3$ and $dV_0$ is the standard volume element of $\mathbb{S}^3$.
\end{trm}

\noindent\emph{Remark}: Our statement might appear slightly different from the one in \cite{Beck}. In \cite{Beck} the right-hand side is replaced by $\frac{1}{12\pi^{2}}\sum_{k=0}^\infty c_{k}\sum_{\ell}|u^{\ell}_{k}|^{2},$ where $c_{k}=\frac{k(k+1)(k+2)}{2}$. But since $\lambda_k=k(k+2)$ one sees that $c_k=\frac{\lambda_k\sqrt{\lambda_k+1}}{2}$. Moreover, in \cite{Beck} the volume element $d\xi$ is the renormalized volume on the sphere, i.e. $d\xi= \frac{1}{2\pi^2}dV_0$.

We will also use the following compactness property.

\begin{prop}\label{propcomp}
For every $p\in [1,\infty)$ the map $\exp: u\mapsto e^u$ sends $H^\frac{3}{2}(\mathbb{S}^3)$ into $L^p(\mathbb{S}^3)$ and is compact.%, i.e. it is continuous and sends bounded sequences into precompact sequences.
\end{prop}

\begin{proof}
 For $u,v\in C^\infty(\mathbb{S}^3)$ we can bound
\begin{equation}\label{stimecont} 
\begin{split}
\|e^u-e^v\|_{L^p}^p&= \int_0^1\frac{d}{dt} \int_{\mathbb{S}^3}[e^u-e^{tv+(1-t)u}]^p dV_0 dt\\
&=\int_{\mathbb{S}^3}\left(\int_0^1p[e^u-e^{tv+(1-t)u}]^{p-1}e^{tv+(1-t)u} dt\right)(u-v)dV_0 \\
&\leq \|u-v\|_{L^{2}}\left(\int_{\mathbb{S}^3}\left(\int_0^1 p\left[e^u-e^{tv+(1-t)u}\right]^{p-1}e^{tv+(1-t)u}dt\right)^{2}dV_{0}\right)^{\frac{1}{2}}\\
&\leq C(p)\|u-v\|_{L^{2}} \left( \int_0^1 \int_{\mathbb{S}^3} e^{2(p-1)u+2tv+2(1-t)u}+e^{2p[tv+(1-t)u]}dV_{0}dt\right)^\frac{1}{2}\\
&\leq C(p,\|u\|_{H^{\frac{3}{2}}},\|v\|_{H^{\frac{3}{2}}}) \|u-v\|_{L^{2}},
\end{split}
\end{equation}
where the last inequality follows from Theorem \ref{beck}. Since $C^\infty(\mathbb{S}^3)$ is dense in $H^\frac32 (\mathbb{S}^3)$ (this follows immediately from the spherical harmonics decomposition), it is easy to see that \eqref{stimecont} actually holds for arbitrary $u,v\in H^\frac32 (\mathbb{S}^3)$. Then continuity of $u\mapsto e^u$ from $H^\frac 32(\mathbb{S}^3)$ into $L^p(\mathbb{S}^3)$ follows.

For the compactness we first notice that
$$\|\nabla e^{u}\|_{L^{1}} = \|\nabla u e^{u}\|_{L^{1}} \leq \|\nabla u\|_{L^{2}}\|e^{u}\|_{L^{2}}.$$
Then by Theorem \ref{beck}, the boundedness of $\|u\|_{H^{\frac{3}{2}}}$ implies the boundedness of $\|\nabla e^{u}\|_{L^{1}}$. Now we can conclude that the map is compact from $H^\frac{3}{2}(\mathbb{S}^3)$ into $L^1(\mathbb{S}^3 )$ using the compact embedding of $W^{1,1}(\mathbb{S}^3)$ into $L^{1}(\mathbb{S}^3)$. If we replace $u$ by $pu$, we have the compactness into $L^{p}(\mathbb{S}^3)$.
\end{proof}

\begin{defin}\label{weaksol}A weak solution $u\in H^\frac32(\mathbb{S}^3)$ of $$P^{3}_{g_{0}}u=f$$ for $f\in H^{-\frac{3}{2}}(\mathbb{S}^3)$ the dual of $H^{\frac{3}{2}}(\mathbb{S}^3)$, is a function $u\in H^{\frac{3}{2}}(\mathbb{S}^3)$ satisfying 
\begin{equation}\label{eq:weak soln}
\int_{\mathbb{S}^3}\left(P^{3}_{g_{0}}\right)^{\frac{1}{2}}u\left(P^{3}_{g_{0}}\right)^{\frac{1}{2}}\varphi dV_{0}=\langle f,\varphi \rangle
\end{equation}
for every $\varphi \in H^{\frac{3}{2}}(\mathbb{S}^3)$, where $\langle\cdot,\cdot \rangle$ is the duality bracket.
\end{defin}

\begin{prop}\label{propH3}
Let $u\in H^{\frac{3}{2}}(\mathbb{S}^3)$  be a weak solution of
\begin{equation}\label{eqCN}
P^{3}_{g_{0}}u+f=ge^{3u},
\end{equation}
where $f\in L^2(\mathbb{S}^3)$ and $g\in L^p(\mathbb{S}^3)$ for some $p>2$. Then $u\in H^3(\mathbb{S}^3)$.
\end{prop}

\begin{proof} By the Beckner's inequality \eqref{eq:beck}, we have $e^{3u}\in L^{\frac{2p}{p-2}}(\mathbb{S}^3)$. Then $ge^{3u}\in L^2(\mathbb{S}^3)$. Hence, $P^3_{g_0}u\in L^2(\mathbb{S}^3)$, which is equivalent to $u\in H^{3}(\mathbb{S}^3)$, as clear from \eqref{def}.
\end{proof}

%\noindent\textbf{Remark.} The assumption $g\in L^p(\mathbb{S}^3)$ for some $p>2$ in Proposition \ref{propH3} can of course be replaced by $g\in L^2(\mathbb{S}^3)$, but the proof would require the $L^p$ estimates, which are considerably more involved.

\section{Existence of non-spherical solutions}\label{s:nonex}
In this section we will prove Theorem \ref{nonstandard}. The proof will follow the ideas in \cite{CC}, and will be a simple consequence of the following theorem about the prescribed $Q$-curvature on $\R{3}$. 

\begin{trm}\label{trmexist}
Assume that $K\in L^\infty(\R{3})$ is positive and satisfies
\begin{equation}\label{assumeK}
K(x)=O(|x|^{-s}) \text{ as }|x|\to\infty\quad \text{for some }s>0.
\end{equation}
Then for
\begin{equation}\label{defmu}
\mu\in \left(0,\min\left\{\frac{s}{6},1\right\}\right),
\end{equation}
 the problem
\begin{equation}\label{eqK}
(-\Delta)^\frac{3}{2}w=Ke^{3w}\ \mbox{in}\ \R{3},\quad \int_{\R{3}}Ke^{3w}dx=2(1-\mu)|\mathbb{S}^3|
\end{equation}
has at least one solution $w$ (in the sense of Definition \ref{deflapl}). Moreover, $w\in H^3_{\loc}(\R{3})$.
\end{trm}

\begin{proof}
Consider
\[
w_0(x)=\log\left(\frac{2}{1+|x|^2}\right),
\]
which is a spherical solution of \eqref{eq0}. Set $d V_\mu= e^{-3\mu w_{0}\circ \pi }dV_{0}$ and $\tilde K=K\circ\pi$, where $\pi$ is the stereographic projection, and consider the functional 
\[
J(w):=\int_{\mathbb{S}^3}\left(\frac{1}{2}|(P^3_{g_0})^{\frac{1}{2}}w|^{2}+ 2(1-\mu)w \right)dV_0-\frac{(1-\mu)4\pi^2}{3}\log\left(\int_{\mathbb{S}^3}\tilde{K}e^{3w} dV_\mu\right).
\]
Notice that $J(w)$ is well-defined on $H^\frac{3}{2}(\mathbb{S}^3)$ since \eqref{defmu} yields
\begin{equation}\label{boundC1}
|\tilde{K}e^{-3\mu w_0\circ\pi }|\leq C_1,
\end{equation}
and thus,
\begin{equation}\label{stimalog}
\begin{split}
&\frac{(1-\mu)4\pi^2}{3}\log\left(\int_{\mathbb{S}^3}\tilde{K}e^{3w} dV_\mu\right)\\
&\quad\le \frac{(1-\mu)4\pi^2}{3}\left(\log\left(\dashint_{\mathbb{S}^3}e^{3(w-\overline w)} dV_0\right)+3\overline w+\log(2\pi^2C_1)\right)\\
&\quad\le \frac{1-\mu}{2}\int_{\mathbb{S}^3}|(P^3_{g_0})^{\frac{1}{2}}w|^{2}dV_0 +(1-\mu)4\pi^2\overline w +C,
\end{split}
\end{equation}
where Beckner's inequality \eqref{eq:beck} is used in the last inequality. Since $J(w+c)=J(w)$ for every $c\in \R{}$,  we can choose a minimizing sequence $\{w_k\}\subset H^\frac{3}{2}(\mathbb{S}^3)$ such that
\begin{equation}\label{zeroav}
\overline w_k=\dashint_{\mathbb{S}^3}w_kdV_0= 0.
\end{equation}
We will show that $\{w_k\}$ is bounded in $H^{\frac 32}(\mathbb{S}^3)$. 
From \eqref{stimalog} and \eqref{zeroav} we obtain
%$$\log\left(\int_{\mathbb{S}^3}\tilde{K}e^{3w} dV_\mu \right)\leq \log\left(\dashint_{\mathbb{S}^3}e^{3w}dV_{0}\right) + \log(2\pi^{2}C_{1}).$$
%Then we can apply Beckner's inequality (Theorem \ref{beck}) to get
%\[
%\frac{8\pi^2}{3}\log\left(\dashint_{\mathbb{S}^3}e^{3w_k} dV_0 \right)\leq \int_{\mathbb{S}^3} \big|(P^{3}_{g_{0}})^\frac{1}{2}w_{k}\big|^2dV_{0},
%\]
%and we have 
\[
\frac{\mu}{2} \int_{\mathbb{S}^3} \big|(P^{3}_{g_{0}})^\frac{1}{2}w_{k}\big|^2dV_{0} \leq J(w_k)+C.
\]
With the Poincar\'e inequality
$$\|w_k\|_{L^2}\le \|w_k\|_{\dot H^{3/2}}=\|(P^3_{g_0})^{1/2}w_k\|_{L^2},$$
which follows easily from \eqref{zeroav} and the spherical harmonics decomposition, we conclude that $\{w_k\}$ is bounded in $H^\frac{3}{2}(\mathbb{S}^3)$.
% to conclude that $\{w_k\}$ is bounded in $H^{\frac 32}_\mu(\mathbb{S}^3)$. In fact, 
%$$\int_{\mathbb{S}^3} (u-\overline{u})^{2} e^{-3\mu w_{0}}dV_{0} \leq \left(\int_{\mathbb{S}^3} (u-\overline{u})^{4}dV_{0}\right)^\frac{1}{2}\left(\int_{\mathbb{S}^3} e^{-6\mu w_{0}}dV_{0}\right)^\frac{1}{2}.$$
%Since $\mu<\frac{1}{2}$, we have that $e^{-6\mu w_{0}}=O((1-x_{4})^{-6\mu})$. Thus, 
%$$\left(\int_{\mathbb{S}^3} e^{-6\mu w_{0}}dV_{0}\right)^\frac{1}{2}<\infty.$$
%Now using Beckner's inequality, we have 
%$$\int_{\mathbb{S}^3} (u-\overline{u})^{4}dV_{0}=||u-\overline{u}||_{L^{4}(\mathbb{S}^3)}^{4}\leq C \left(\int_{\mathbb{S}^3} |(P^{3}_{g_{0}})^{\frac{1}{2}}u|^{2} dV_{0}\right)^{2}$$
%for every $u\in H^{\frac{3}{2}}_{\mu}(\mathbb{S}^3)\subset H^{\frac{3}{2}}(\mathbb{S}^3)$.
Hence, it has a subsequence weakly converging in $H^\frac{3}{2}(\mathbb{S}^3)$ to a minimizer $u$. Indeed, up to a subsequence
$$\lim_{k\to\infty}\log\left(\int_{\mathbb{S}^3}\tilde Ke^{3w_k} dV_\mu\right) =\log\left(\int_{\mathbb{S}^3}\tilde K e^{3u} dV_\mu\right)$$
by \eqref{boundC1} and the compactness of $w\mapsto e^w$ from $H^{\frac{3}{2}}(\mathbb{S}^3)$ to $L^{p}(\mathbb{S}^3)$ for every $p>1$ in Proposition \ref{propcomp}. Moreover, by convexity and weak convergence we have
$$\liminf_{k\to\infty}\int_{\mathbb{S}^3}\left(\frac{1}{2}|(P^3_{g_0})^{\frac{1}{2}}w_k|^{2}+ 2(1-\mu) w_k \right)dV_0 \ge \int_{\mathbb{S}^3}\left(\frac{1}{2}|(P^3_{g_0})^{\frac
{1}{2}}u|^{2}+ 2(1-\mu) u \right)dV_0.$$
This shows that $u$ is a minimizer of $J$.
In particular, $u$ is a weak solution of 
\begin{equation}\label{umin}
P^3_{g_0}u+2(1-\mu)=\frac{(1-\mu)4\pi^2\tilde{K}e^{-3\mu w_0\circ\pi}e^{3u}}{\int_{\mathbb{S}^3}\tilde{K}e^{3u}\ud V_\mu},
\end{equation}
in the sense of Definition \ref{weaksol}.
Choose a constant $C$ such that $\tilde u:=u+C$ satisfies
\[
\int_{\mathbb{S}^3}\tilde{K}e^{3\tilde u} d V_\mu=(1-\mu)4\pi^2.
\]
Then $\tilde u$ solves
\[
P^3_{g_0}\tilde u+2(1-\mu)=\tilde{K}e^{-3\mu w_0\circ\pi}e^{3\tilde u}.
\]
%Now recall that $P^3_{\mu}=e^{3\mu w_0\circ\pi}P^{3}_{g_{0}}$, hence we have
%$$e^{3\mu w_{0}}P^{3}_{g_{0}}\tilde u+K_\mu(x)=\tilde{K}(x)e^{3\tilde u}.$$
%Thus, $w=\tilde{u}-\mu w_{0}\circ \pi$ satisfies
%$$P^{3}_{g_{0}}w+1=\tilde{K}e^{3w},$$
By \eqref{boundC1} we know that $\tilde{K}e^{-3\mu w_0\circ\pi}\in L^\infty(\mathbb{S}^3)$. Hence, $\tilde{u}\in H^3(\mathbb{S}^3)$ by Proposition \ref{propH3}. It follows from Lemma \ref{lemmaconf} below and
$$\int_{\R{3}}Ke^{3w}dx=\int_{\mathbb{S}^3}\tilde K e^{3\tilde u}dV_\mu =(1-\mu)4\pi^2$$
that $w:=\tilde{u}\circ \pi^{-1}+(1-\mu)w_0\in H^3_{\loc}(\R{3})$ is a solution of \eqref{eqK}.
\end{proof}

\begin{lemma}\label{lemmaconf}
If $\pi$ is the stereographic projection from $\mathbb{S}^3\setminus\{p\}$ to $\mathbb{R}^{3}$, then the pull back of the operator $(-\Delta)^{\frac{3}{2}}$ under $\pi$ is the operator $P^{3}_{g_{0}}$. More precisely, if $u\in H^3(\mathbb{S}^3)$, then
\begin{equation}\label{eqpullback}
(P^{3}_{g_{0}}u)\circ \pi^{-1}=e^{-3w_{0}}(-\Delta)^{\frac{3}{2}}(u\circ \pi^{-1}),
\end{equation}
in the sense of tempered distributions.
\end{lemma}

\begin{proof}
We know from \cite{Bra2} that \eqref{eqpullback} holds for $u\in C^{\infty}(\mathbb{S}^3)$. For $u\in H^{3}(\mathbb{S}^3)$ it follows from standard approximations. Notice first that $\Delta (u\circ \pi^{-1}) \in L^{2}(\R{3})$. Indeed, if we set $U=u\circ \pi^{-1}$, then  
$$|\nabla^{2}U|^2\leq C \left(|(\nabla^{2}u)\circ \pi^{-1}|^2 \frac{1}{(1+|x|^{2})^4} + |(\nabla u) \circ \pi^{-1}|^2 \frac{1}{(1+|x|^{2})^{3}}\right).$$
Therefore, if we consider the first term on the right-hand side, we get with H\"older's inequality and Sobolev's embedding
$$\int_{\mathbb{R}^{3}} \left(|(\nabla^{2}u)\circ \pi^{-1}|^2 e^{\frac{3}{2}w_{0}} \right) \left(e^{-\frac{3}{2}w_{0}} (1+|x|^{2})^{-4}\right) dx \leq C \|\nabla^{2} u\|^{2}_{L^{4}}\le C\|u\|_{H^3}^2.$$
A similar inequality holds for the second term and we get 
\begin{equation}\label{stimaU}
\|\Delta (u\circ\pi^{-1})\|_{L^2(\R{3})}\le C\|u\|_{H^{3}(S^3)}.
\end{equation}
Since $\Delta U \in L^{2}(R^{3})\subset L_{1/2}(\R{3})$ (by H\"older's inequality),  $(-\Delta)^{\frac{1}{2}}(-\Delta U)$ is well defined.

Now pick a sequence $\{u_{k}\}\subset C^{\infty}(\mathbb{S}^3)$ such that $u_{k} \to u$ in $H^{3}(\mathbb{S}^3)$. By  \eqref{eqpullback} we have 
$$(P^{3}_{g_{0}}u_{k})\circ\pi^{-1}=e^{-3w_{0}}(-\Delta)^{\frac{3}{2}}(u_{k}\circ \pi^{-1}).$$
The left hand side converges to $(P^{3}_{g_0}u)\circ \pi^{-1}$ in the sense of tempered distribution, since $P^{3}_{g_{0}}u_{k}\to P^{3}_{g_{0}}u$ in $L^2(\mathbb{S}^3)$. On the other hand, \eqref{stimaU} implies that 
$$
\Delta (u_{k}\circ \pi^{-1})\to \Delta (u\circ \pi^{-1})\quad \text{in } L^{3}(\R{3})\text{ and hence in }L_{1/2}(\R{3}),
$$
which implies $(-\Delta)^{\frac{3}{2}}(u_{k}\circ \pi^{-1})\to (-\Delta)^{\frac{3}{2}}(u\circ \pi^{-1})$ in $\mathcal{S}'(\R{3})$. Since $e^{-3w_0}\varphi\in \mathcal{S}(\R{3})$ for every $\varphi\in \mathcal{S}(\R{3})$, we also have  $e^{-3w_0}(-\Delta)^{\frac{3}{2}}(u_{k}\circ \pi^{-1})\to e^{-3w_0} (-\Delta)^{\frac{3}{2}}(u\circ \pi^{-1})$ in  the sense of tempered distributions and the proof is complete.
\end{proof}
% In particular $u \in W^{2,1}(\mathbb{S}^3)$ thus
%$$\int_{\R{3}} \frac{|\Delta u(x)|}{1+|x|^{4}}dx\leq C\left(\int_{\R{3}} \frac{|\Delta u(x)|^{p}}{(1+|x|%^{2})^{3}}dx\right)^{\frac{1}{p}}\left(\int_{\R{3}} \frac{1}{(1+|x|^{2})^{\frac{2p-3}{p-1}}}dx\right)^{\frac{p-1}{p}}$$
%But $$\int_{\R{3}} \frac{|\Delta u(x)|^{p}}{(1+|x|^{2})^{3}}dx\leq ||\Delta u||_{L^{p}(\mathbb{S}^3)}^{p}<\infty$$
%and for $p>5$,
 %$$\int_{\R{3}} \frac{1}{(1+|x|^{2})^{\frac{2p-3}{p-1}}}dx<\infty.$$
%So the solutions obtained by this variational method coincides with the setting defined above.\\

\begin{proof}[Proof of Theorem \ref{nonstandard}]
Choose $K(x)=2e^{-3a|x|^2}$ for some $a>0$. It is clear that $K$ satisfies the assumptions in Theorem \ref{trmexist} for any positive $s$. Fix $\mu=1-\frac{V}{|\mathbb{S}^3|}\in (0,1)$ and let $w$ be the corresponding solution of \eqref{eqK}. Since $(-\Delta)^\frac{3}{2}(-a|x|^2)=0$, we have that $u:=w-a|x|^2$ is a solution of \eqref{eq0}. Moreover, 
\[
\int_{\R{3}}e^{3u}dx=\frac{1}{2}\int_{\R{3}}Ke^{3w}dx=(1-\mu)|\mathbb{S}^3|=V.
\]
Thus, \eqref{area} is satisfied. Finally, by noticing that $u\in H^3_{\loc}(\R{3})\hookrightarrow C^{1,\alpha}_{\loc}(\R{3})$ for some $\alpha>0$, we have $(-\Delta)^\frac{3}{2}u=2e^{3u}\in C^{1,\alpha}_{\loc}(\R{3})$. By the Schauder estimates for fractional Laplacian equations (Corollary \ref{corschauder} in the appendix), $\Delta u\in C^{2,\alpha}_{\loc}(\R{n})$, and thus, $u\in C^{4,\alpha}_{\loc}(\R{n})$ by the classical Schauder estimates. In particular, $e^{3u}\in C^{4,\alpha}_{\loc}(\R{n})$. By the bootstrap procedure, we have that $u\in C^\infty(\R{3})$.
\end{proof}

\begin{rem}
In Theorem \ref{trmexist}, if one additionally assumes that $K$ is radially symmetric, 
one can prove the existence of a radially symmetric solution of \eqref{eqK}. Indeed, it suffices to minimize $J$ among rotationally symmetric functions only. 
Since $J$ is invariant under rotations, the minimizer will be a critical point of $J$ in all of $H^\frac32 (\mathbb{S}^3)$, i.e., it solves \eqref{umin}, see e.g. \cite{Sma}. 
Since we chose $K(x)=2e^{-3a|x|^2}$ in the proof of Theorem \ref{nonstandard}, 
it follows that for $V\in (0,2\pi^2)$ we can find a solution to \eqref{eq0}-\eqref{area} which is radially symmetric. 
Taking the results of \cite{HM} and \cite{W-Y} into account, we expect that for each $V\in (0,2\pi^2)$ there are many non-radially symmetric solutions to Problem \eqref{eq0}-\eqref{area}.
\end{rem}

\section{Estimates and technical lemmas}\label{s:estimates}

In this section, we establish some estimates for smooth solutions $u$ of \eqref{eq0}-\eqref{area}.
\begin{lemma}\label{upper bound}
Let $u$ be a smooth solution of \eqref{eq0}-\eqref{area} and $v$ be as in \eqref{eqv}. Then there exists a positive constant $C$ such that for $|x|\ge 4$,
\[
-v(x)\leq \alpha\log|x| +C.
\]
\end{lemma}
\begin{proof}
This is a special case of Lemma 9 in \cite{mar1} (originally proven in dimension $4$ in \cite[Lemma 2.1]{lin1}).
\end{proof}

\begin{lemma}\label{equal} 
Let $u$ be a smooth solution of \eqref{eq0}-\eqref{area} and $v$ be as in \eqref{eqv}.  Then $\Delta v\in L_{1/2}(\R{3})$  and 
$$(-\Delta)^{\frac32} v=(-\Delta)^{\frac32}u=2e^{3u}.$$
\end{lemma}

\begin{proof} Differentiating under the integral in \eqref{eqv} we obtain
$$-\Delta v(x)=-\frac{1}{2\pi^2}\int_{\R{3}}\frac{f(y)}{|x-y|^2}dy,$$
where $f:= 2e^{3u}\in L^1(\R{3})$. 
Then the conclusion follows at once from Lemma \ref{lemmafund}.
\end{proof}

\begin{lemma}\label{lem:constant} 
Let $w\in L_{1/2}(\R{3})$ satisfy $(-\Delta)^{\frac12} w=0$ in $\R{3}$. Then $w$ is a constant.
\end{lemma}

\begin{proof}
The lemma follows from the estimates for $w$ and a scaling argument.
By Proposition \ref{propschauder} in the appendix, we have 
\[
\|\nabla ^2 w\|_{L^\infty(B_1)}\le C\int_{\R{3}}\frac{|w(x)|}{1+|x|^4} dx,
\]
where $C>0$ is a universal constant. Given $x\in\R{3}$, we choose $r>|x|$ and set $w_r(y):=w(ry)$. Then,
\[
|\nabla w(x)-\nabla w(0)|=\frac {1}{r}\left|\nabla w_r\left(\frac{x}{r}\right)-\nabla w_r(0)\right|\le \frac{|x|}{r^{2}}\|\nabla^2 w_r\|_{L^\infty(B_1)}.
\]
Since $(-\Delta)^\frac{1}{2} w_r=0$ in $\R{3}$, we have
\[
\|\nabla^2 w_r\|_{L^\infty(B_1)}\le C\int_{\R{3}}\frac{|w_r(x)|}{1+|x|^4} dx=C r\int_{\R{3}}\frac{|w(x)|}{r^4+|x|^4} dx.
\]
Thus,
\[
|\nabla w(x)-\nabla w(0)|\le C \frac{|x|}{r}\int_{\R{3}}\frac{|w(x)|}{r^4+|x|^4} dx\to 0\quad \text{as }r\to\infty.
\]
Then $\nabla w(x)=\nabla w(0)$. Since $x$ was arbitrary, $\nabla w$ is a constant and $w$ is an affine function. On the other hand, it is clear that the only affine functions in $L_{1/2}(\R{3})$ are the constant functions.
\end{proof}

%\paragraph{A Liouville theorem} In analogy with Theorems 5-6 from \cite{mar1} we should be able to prove

\begin{prop}\label{liou} 
Let $u$ be a smooth solution of \eqref{eq0}-\eqref{area} and $v$ be as \eqref{eqv}. Let $p= u-v$. Then $p$ is a polynomial of degree $0$ or $2$. Moreover, $\Delta p\le 0$ and $\sup_{\R{3}}p<\infty$.
\end{prop}

\begin{proof}
It follows from Lemma \ref{equal} that $(-\Delta)^\frac{3}{2} p=0$ in $\R{3}$. By Lemma \ref{lem:constant}, $\Delta p$ is a constant, and in particular $\Delta^2 p\equiv 0$.  Taking Lemma \ref{upper bound} and \eqref{area} into account it then follows from a generalization of Liouville's theorem (see e.g. Theorem 6 in \cite{mar1}) that $p$ is a polynomial of degree at most $2$. Since $u$ satisfies \eqref{area}, then $p$ can not be of degree $1$, particularly in view of Lemma \ref{upper bound}. The claim $\sup_{\R{3}} p<\infty$ follows from Lemma 11 in \cite{mar1}.

It remains to show that $\Delta p\le 0$ in $\R3$. We shall adapt some arguments from the proof of Lemma 2.2 in \cite{lin1}. By Pizzetti's formula (see e.g. (10) in \cite{mar1}) we have for any $x_0\in \R{3}$ and $r>0$,
\[
 \frac{r^2}{6}\Delta p(x_0)=\dashint_{\de B_r(x_0)}p d\sigma-p(x_0),
\]
where $\dashint$ denotes the average.
Hence by Jensen's inequality,
\[
 \begin{split}
  \exp\left(\frac{r^2}{2}\Delta p(x_0)\right)%&\leq e^{-3h(x_0)}\exp(3\dashint_{|x-x_0|=r}hd\sigma)\\
&\leq e^{-3p(x_0)} \dashint_{\de B_r(x_0)}e^{3p}d\sigma\\
%&\leq e^{-3h(x_0)} \dashint_{|x-x_0|=r}e^{3u}e^{-3v}d\sigma\\
&\leq C e^{-3p(x_0)} r^{3C} \dashint_{\de B_r(x_0)}e^{3u}d\sigma\\
&\leq C e^{-3p(x_0)} r^{3C-2} \int_{\de B_r(x_0)}e^{3u}d\sigma,
 \end{split}
\]
where the estimate of $v$ is used in the second inequality and $C$ is a constant independent of $r$. Integrating with respect to $r$ it follows that
$$r^{2-3C}\exp\left(\frac{r^2}{2}\Delta p(x_0)\right)\in L^1([0,\infty)).$$
Hence, $\Delta p(x_0)\leq 0$.  
\end{proof}

%\begin{lemma}\label{deg p} We have $u=v+p$ where $p$ is a polynomial of degree $0$ or $2$ and $\sup_{\R{3}} p<\infty.$
%\end{lemma}

%\begin{proof} We have $(-\Delta)^\frac{3}{2}(u-v)=0$ and, according to Lemma \eqref{upper bound} and \eqref{area}, the function $u-v$ satisfies the hypothesis of Theorem \ref{liou}. Hence $u-v=p$ is a polynomial of degree at most $2$. That is has to be bounded from above follows easily from Lemma \ref{upper bound} and \eqref{area} (if $p$ grows linearly in some direction, then also $u$ grows linearly in that direction, and this will contradict \eqref{area}). Then $p$ cannot be of degree $1$.
%\end{proof}

%\emph{\textbf{Remark}}: Once we have a sign of $\Delta h$, our local estimates established in the paper of fractioinal Nirenberg apply and the rest of arguments of Theorem \ref{clas1} would be similar to those of Chang-Shou Lin's \cite{lin1}. We have not finished yet.

%\medskip

%The proof could follow from the analog of Lemma 3 (Pizzetti's formula) and Proposition 4 (elliptic estimates for polyharmonic functions) in \cite{mar1}. But what would be the analog of Pizzetti's formula???? Maybe some of you already knows some kind of Liouville theorem for $(-\Delta)^\frac{1}{2}$ or $(-\Delta)^\frac{3}{2}$?? Google it?

A consequence of Proposition \ref{liou} is the following.
\begin{cor}\label{cor:lapu}
Let $u$ be a smooth solution of \eqref{eq0}-\eqref{area}. Then
\begin{equation}\label{eq: int}
-\Delta u(x)=\frac{1}{\pi^2}\int_{\R{3}} \frac{e^{3u(y)}}{|x-y|^2}\,dy+a,
\end{equation}
for some constant $a\ge 0$. 
\end{cor}

\begin{lemma} \label{lem: upper bd}
Let $u$ be a smooth solution of \eqref{eq0}-\eqref{area}. Then $0\le-\Delta u(x)\leq A$ in $\R{3}$, where $A>0$ is a constant depending on $u$. Consequently, there exists a constant $B>0$ depending only on $A$ and $V$ such that $u\le B$ in $\R{3}$.
\end{lemma}
\begin{proof}
It follows from Corollary \ref{cor:lapu} that $u$ satisfies \eqref{eq: int}. Then the conclusion follows from Lemmas 3.1 and 3.2 in \cite{Xu}. Note that although the statement of Lemma 3.1 in \cite{Xu} is for solutions of \eqref{eq: int} with $a=0$, its proof still works for solutions of \eqref{eq: int} with the following mild changes. The function $q(x)$ defined after (3.8) of \cite{Xu} is replaced by
$$q(x)=w(x)-h(x) -p(x),$$
where $p(x)$ is the polynomial of degree $2$ defined in Proposition \ref{liou}.
The bound  (3.9) of \cite{Xu}, now is replaced by
$$0\le  -\Delta q \le V+a.$$
The bound $w(y)=q(y)+h(y)\le C+h(y)$ on page 10 of \cite{Xu} is replaced by
$$w(y)=q(y)+h(y) +p(x)\le C+h(y),$$
where we use that $\sup_{\R{3}}p(x)\le C$.
\end{proof}

\begin{lemma} \label{lower bound}
Let $u$ be a smooth solution of \eqref{eq0}-\eqref{area} and $v$ be as in \eqref{eqv}. Then for any $\varepsilon>0$ there exists $R>0$ such that for all $|x|\ge R$,
\[
 -v(x)\geq (\alpha -\varepsilon)\log|x|.
\]
Moreover, \eqref{limDeltav} holds. 
\end{lemma}

\begin{proof}
 As in the proof of Lemma 2.4 of \cite{lin1}, we can show that for any $\ve>0$
there exists $R=R(\ve)>0$ such that
\[
 -v(x)\geq (\alpha -\frac{\ve}{2})\log|x|+\frac{1}{\pi^2}\int_{B_1(x)}\log|x-y| e^{3u(y)}\,\ud y,
\]
where $B_1(x)$ denotes the ball with center $x$ and radius $1$. 
By Lemma \ref{lem: upper bd}, the last term is bounded from below independently of $x$, which implies that $ -v(x)\geq (\alpha -\varepsilon)\log|x|$ for large $|x|$. 

Meanwhile, for $\ell =1,2$
\[
\begin{split}
|\nabla^\ell v(x)|&\le C\int_{\R{3}} \frac{e^{3u(y)}}{|x-y|^\ell}\,d y\\
&=C\int_{B_1(x)} \frac{e^{3u(y)}}{|x-y|^\ell}\,d y+C\int_{\R{3}\setminus B_1(x)} \frac{e^{3u(y)}}{|x-y|^\ell}\,d y
\end{split}
\]
Then we bound
\[
\begin{split}
\int_{B_1(x)} \frac{e^{3u(y)}}{|x-y|}\,d y&\le \int_{B_1(x)} \frac{e^{3u(y)}}{|x-y|^2}\, d y\\
&\le \left(\int_{B_1(x)} \frac{1}{|x-y|^\frac{8}{3}}\,d y\right)^{3/4}\left(\int_{B_1(x)} e^{12u(y)}\,d y\right)^{1/4}\\
&\leq C\left(\int_{B_1(x)} e^{12v(y)+12p(y)}\,d y\right)^{1/4}\to 0 \quad \mbox{as }|x|\to \infty,
\end{split}
\]
 since $ v(x)\leq (-\alpha +\varepsilon)\log|x|$ and $p(x)$ is bounded from above by Proposition \ref{liou}. On the other hand, by the dominated convergence theorem,
\[
\int_{\R{3}\setminus B_1(x)} \frac{e^{3u(y)}}{|x-y|^\ell}\,d y\to 0 \quad \mbox{as }|x|\to \infty,\quad \ell=1,2,
\] 
and \eqref{limDeltav} follows.
\end{proof}

%\paragraph{The polynomial $p$} %The analog of Lemma 9 in \cite{mar1} should be

%\begin{lemma}\label{lemma9} We have
%$$v(x)\ge -2\frac{V}{|\mathbb{S}^3|}\log|x|-C$$
%for $|x|\ge 4$.
%\end{lemma}

%\begin{proof} Identical to the proof of Lemma 9 in \cite{mar1}.
%\end{proof}

In the proof of Theorem \ref{trmvolume} we shall also use the following Pohozaev-type identity, whose proof can be found in \cite[Theorem 2.1]{Xu}.

\begin{lemma}[\cite{Xu}]\label{lemmaxu} Let $w\in C^{1}(\R{3})$ solve the integral equation
\begin{equation}\label{eqxu}
w(x)=\frac{1}{2\pi^2}\int_{\R{3}}\log\left(\frac{|y|}{|x-y|}\right) K(y)e^{3w(y)}dy,
\end{equation}
where $K\in C^1(\R{3})$ and $Ke^{3w}\in L^1(\R{3})$.
Then, setting
$$\alpha:=\frac{1}{2\pi^2}\int_{\R{3}}Ke^{3w}dx,$$
we have
\begin{equation}\label{eqpoho}
\alpha(\alpha-2)=\frac{1}{3\pi^2}\int_{\R{3}}x\cdot\nabla K(x)e^{3w(x)}dx.
\end{equation}
\end{lemma}

\section{Proof of Theorems \ref{clas1}, \ref{clas2} and \ref{trmvolume}}

\begin{proof}[Proof of Theorem \ref{clas1}]
It follows from Lemma \ref{upper bound}, Proposition \ref{liou}, and Lemma \ref{lower bound}.
\end{proof}

\begin{proof}[Proof of Theorem \ref{clas2}]
Clearly \emph{(i)} implies \emph{(ii)}-\emph{(vi)}. In view of \eqref{repr}, \eqref{limDeltav} and \eqref{limv} it is also easy to see that \emph{(ii)}, \emph{(iii)} and \emph{(iv)} are equivalent. 
Now if \emph{(ii)} holds, then $u=v+C$, i.e. $u$ solves an integral equation and Theorem 4.1 in \cite{Xu} implies that $u$ is spherical.

To prove that either \emph{(v)} or \emph{(vi)} imply \emph{(i)} we assume that \emph{(i)} does not hold. Then \emph{(ii)} does not hold and $\deg p=2$. Hence $|\nabla p|^2$ is a polynomial of degree $2$.
Then
$$
R_{g_u}=-2e^{-2u}(2\Delta u+ |\nabla u|^2)= -2e^{-2(v+p)}(2\Delta v + 2\Delta p +|\nabla p|^2 +2 \nabla p\cdot\nabla v +|\nabla v|^2).
$$
It follows from \eqref{limDeltav} at once that $\liminf_{|x|\to \infty} R_{g_u}=-\infty$, so \emph{(vi)} does not hold.
As for \emph{(v)}, if \emph{(i)} fails to hold, then $\deg p =2$ and from \eqref{limv} we infer
$$\liminf_{|x|\to\infty}\frac{u(x)}{|x|^2} < 0.$$
This implies that $\pi^{*}(e^{2u}|dx|^2)$ is either discontinuous or vanishes at the point $(0,0,0,-1)\in \mathbb{S}^3$, and therefore \emph{(vi)} also fails to hold (see \cite{mar1} for more details).

Finally, assuming that $u$ is non-spherical one has that \emph{(ii)} does not hold. So $\Delta p =const \ne 0$ (the case $\Delta p\equiv 0$, together with $\sup_{\R{3}}p<\infty$ would yield $p\equiv const$ by Liouville's theorem), and \eqref{deltaa} follows at once from $\Delta u=\Delta v+\Delta p$ and \eqref{limDeltav}.
\end{proof}

\begin{proof}[Proof of Theorem \ref{trmvolume}] The function $v(x)$ satisfies the integral equation \eqref{eqv}, which can be written as 
\[
v(x)=\frac{1}{2\pi^2}\int_{\R{3}}\log\left(\frac{|y|}{|x-y|}\right) K(y)e^{3v(y)}dy,
\]
where
$$K(x)=2e^{3(u-v)}=2e^{3p(x)}$$
and $p$ is the polynomial given by Theorem \ref{clas1}. Since $u$ is non-spherical, we have that $p$ is not a constant and, up to a translation
$$p(x)=-\sum_{i=1}^3 a_i x_i^2 +c_0$$
for some coefficients $a_i\ge 0$ not all vanishing. In particular
$$x\cdot \nabla p(x)\le 0,\quad x\cdot \nabla p(x)\not\equiv 0.$$
This of course implies
$$x\cdot \nabla K(x)\le 0,\quad x\cdot \nabla K(x)\not\equiv 0.$$
It follows from \eqref{eqpoho} that $\alpha<2$, i.e.
$$2>\frac{1}{|\mathbb{S}^3|}\int_{\R{3}}2e^{3p}e^{3v}dx=\frac{2}{|\mathbb{S}^3|}\int_{R^3}e^{3u}dx=\frac{2V}{|\mathbb{S}^3|}.$$
\end{proof}

\appendix
\section{The fractional Laplacian in $\R{n}$}\label{s:frac}

If $\sigma\in (0,1)$ and $u$ belongs to the Schwarz space $\mathcal S$ of rapidly decreasing smooth functions in $\R{n}$, then $(-\Delta)^\sigma u$ is defined by
$$\widehat{(-\Delta)^\sigma u}(\xi)=|\xi|^{2\sigma}\hat u(\xi),$$
where
$$\hat{f}(\xi)=\M{F}(f)(\xi):=\frac{1}{(2\pi)^{n/2}}\int_{\R{n}}f(x)e^{- ix\cdot \xi}dx$$
denotes the Fourier transform.
An equivalent definition is the following:
\begin{equation}\label{fraclapl}
(-\Delta)^\sigma u(x):=C_{n,\sigma} P.V.\int_{\R{n}}\frac{u(x)-u(y)}{|x-y|^{n+2\sigma}}dy,
\end{equation}
where the right-hand side is defined in the sense of the principal value. One can see that \eqref{fraclapl} makes sense for classes of functions larger than the Schwarz space, for instance for functions in $C^{2\sigma+\alpha}_{\loc}(\R{n})\cap L_{\sigma}(\R{n})$ for some $\alpha>0$, where
$$L_{\sigma}(\R{n}):=\left\{u\in L^1_{\loc}(\R{n}):\int_{\R{n}}\frac{|u(x)|}{1+|x|^{n+2\sigma}}dx<\infty   \right\},$$
and $C^{2\sigma+\alpha}_{\loc}(\R{n}):=C^{0,2\sigma+\alpha}_{\loc}(\R{n})$ for $2\sigma+\alpha\le1$ and  $C^{2\sigma+\alpha}_{\loc}(\R{n}):=C^{1,2\sigma-1+\alpha}_{\loc}(\R{n})$ for $2\sigma+\alpha>1$. We denote $\|u\|_{L_{\sigma}(\R{n})}=\int_{\R{n}}\frac{|u(x)|}{1+|x|^{n+2\sigma}}dx$. 
Observing that
\begin{equation}\label{stimaphi}
\sup_{\R{n}}(1+|x|^{n+2\sigma})|(-\Delta)^\sigma \varphi(x)|< +\infty\quad \text{for }\varphi\in \mathcal{S},
\end{equation}
and that $(-\Delta)^{\sigma}:\mathcal{S}\to\mathcal{S}$ is symmetric, 
as shown in \cite{Sil}, one can define $(-\Delta)^{\sigma} u$ by duality for functions $u\in L_{\sigma}(\R{n})$ as a tempered distribution via the relation
\begin{equation}\label{lapldual}
\langle (-\Delta)^{\sigma} u,\varphi\rangle =\int_{\R{n}} u(x) (-\Delta)^{\sigma} \varphi(x)dx\quad\text{for every }\varphi\in\mathcal{S}.
\end{equation}
%where $\mathcal{S}$ denotes the Schwarz space of rapidly decreasing smooth functions in $\R{3}$.
That for $u\in C^{2\sigma+\alpha}_{\loc}(\R{n})\cap L_\sigma (\R{n})$ the definitions \eqref{fraclapl} and \eqref{lapldual} coincide is shown in \cite[Proposition 2.4]{Sil}.

The following lemma is well-known, but we include a proof here for convenience and completeness.
\begin{lemma}\label{lemmafund} The function $K(x):=\frac{1}{2\pi^2|x|^2}$ is a fundamental solution of $(-\Delta)^\frac{1}{2}$ in $\R{3}$ in the sense that for every $f\in L^1(\R{3})$ we have $K*f\in L_{1/2}(\R{3})$ and
\begin{equation}\label{K*f}
(-\Delta)^\frac{1}{2} (K*f)=f,
\end{equation}
in the sense of \eqref{lapldual}.
\end{lemma}

\begin{proof} First of all, it follows easily from Theorem 5.9 in \cite{LL} that \eqref{K*f} holds if we assume $f\in \mathcal C^\infty_c(\R{3})$. 

Secondly, we notice that, if $f\in L^{1}$ then $K*f\in L_{1/2}(\R{3})$. Indeed,
$$K(x)=\frac{1}{2\pi^2|x|^{2}}\chi_{B_{1}} + \frac{1}{2\pi^2|x|^{2}}\chi_{\mathbb{R}^{3}\setminus B_{1}}=: K_1(x)+K_2(x),$$
and $K_{1}\in L^{\frac{3}{2}-\varepsilon}(\mathbb{R}^{3})$,  $K_{2}\in L^{\frac{3}{2}+\varepsilon}(\mathbb{R}^{3})$ for any $\varepsilon>0$. Hence, by Young's inequality
$$K*f\in L^{\frac{3}{2}-\varepsilon}(\mathbb{R}^{3})+L^{\frac{3}{2}+\varepsilon}(\mathbb{R}^{3})\subset L_{1/2}(\R{3}),$$
where the last inclusion follows from H\"older's inequality.

Lastly, if $f\in L^1(\R{3})$ we take a sequence $(f_k)\subset C^\infty_c(\R{3})$ with $f_k\to f$ in $L^1(\R{3})$. Then for every $\varphi\in\mathcal S$ we have
$$(I)_k:=\langle (-\Delta)^\frac{1}{2}(K*f_k),\varphi\rangle =\langle f_k,\varphi\rangle \to \langle f,\varphi\rangle\quad \text{as }k\to\infty.$$
Since $K*f=K_{1}*f+K_{2}*f\in L^{\frac{3}{2}-\varepsilon}+L^{\frac{3}{2}+\varepsilon}$, we have $K_{1}*f_{k}\to K_{1}*f$ in $L^{\frac{3}{2}-\varepsilon}$, and thus, in $L_{1/2}(\R{3})$ by H\"older's inequality. Similarly, $K_{2}*f_{k}\to K_{2}*f$ in $L_{1/2}(\R{3})$, and  $K*f_k\to K*f$ in $L_{1/2}(\R{3})$. By \eqref{stimaphi}, we find
$$(I)_k=\langle K*f_k,(-\Delta)^\frac{1}{2}\varphi\rangle \to \langle K*f,(-\Delta)^\frac{1}{2}\varphi\rangle.$$
%\begin{equation}\label{stimapsi}
%|\psi(x)|\le \frac{C}{1+|x|^4}.
%\end{equation}
%It remains to show that
%$$\frac{1}{2\pi^2} \int_{\R{3}} \int_{\R{3}}\frac{f_k(y) -f(y)}{|x-y|^2}dy\, \psi(x)dx\to 0\quad \text{as }k\to\infty.$$
%Define $g_k:=f_k-f\to 0$ in $L^1(\R{3})$.
%Then using that $\|h_1*h_2\|_{L^1}\le \|h_1\|_{L^1}\,\|h_2\|_{L^1}$ we have
%$$\left|\int_{B_1(x)}\frac{g_k(y)}{|x-y|^2}dy\right|\le C\|g_k\|_{L^1(\R)},$$
%and using that for $|x-y|\ge 1$ we have $\frac{1}{|x-y|^2}\le 1$ we bound
%$$\left|\int_{\R{3}\setminus B_1(x)}\frac{g_k(y)}{|x-y|^2}dy\right|\le \|g_k\|_{L^1}.$$
%Therefore, taking \eqref{stimapsi} into account, we see that
%$$(I)_k\to \left\langle  \frac{1}{2\pi^2}\int_{\R{3}}\frac{ f(y)}{|x-y|^2}dy,(-\Delta)^\frac{1}{2}\varphi \right\rangle\quad \text{as }k\to\infty,$$
Hence, we conclude that $(-\Delta)^\frac{1}{2} (K*f)= f$ in the sense of \eqref{lapldual}.
\end{proof}

\subsection{Schauder estimates}

The following proposition should be well-known, but we include here an elementary proof of the estimate \eqref{schaud2} which was used in Section \ref{s:estimates}.

Let $\Omega$ be a domain in $\R{n}$ and $f\in L^1(\Omega)$. We say that $u\in L_{\sigma}(\R{n})$ is a solution of $(-\Delta)^\sigma u=f$ in $\Omega$ if
\[
\int_{\R{n}}  u\,  (-\Delta)^\sigma \varphi  dx=\int_{\R{n}}  f\,   \varphi  dx\quad \text{for every }\varphi\in C_c^{\infty}(\Omega).
\]

\begin{prop}\label{propschauder} If $u\in L_{\sigma}(\R{n})$ for some $\sigma\in (0,1)$ and $(-\Delta)^\sigma u=0$ in $B_{2r}$ for some $r>0$, then $u\in C^\infty(B_{2r})$. Moreover, for every $k\in \mathbb{N}$ the following estimate holds:
%\begin{equation}\label{schaud1}
%[\nabla^k u]_{C^{0,\alpha}(B_r)}\le C_{n,\sigma,k,\alpha, r}\|u\|_{L_\sigma},
%\end{equation}
%or more precisely
\begin{equation}\label{schaud2}
\|\nabla^k u\|_{L^\infty(B_r)}\le \frac{C_{n,\sigma,k}}{r^{k}} \left( r^{2\sigma}\int_{\R{n}\setminus B_{2r}}\frac{|u(x)|}{|x|^{n+2\sigma}}dx + \frac{\|u\|_{L^1(B_{2r})}}{r^n}\right),
\end{equation}
where $C_{n,\sigma,k}$ is a positive constant depending only on $n,\sigma$ and $k$.
%and
%\begin{equation}\label{schaud3}
%[\nabla^k u]_{C^{0,\alpha}(B_r)}\le \frac{C_{n,\sigma,k,\alpha}}{r^{k+\alpha}} \left( r^{2\sigma}\int_{\R{n}\setminus B_{2r}}\frac{|u(x)|}{|x|^{n+2\sigma}}dx + \frac{\|u\|_{L^1(B_{2r})}}{r^n}\right).
%\end{equation}
\end{prop}

Notice that the right-hand sides of \eqref{schaud2} are equivalent to $C_{n,\sigma,k,\alpha,r}\|u\|_{L_\sigma}$ for every fixed $r$ and, although this term is more compact, it is not scale invariant with respect to $r$.

For the proof of this proposition we will use a couple of results from \cite{Sil}. Following the notations of Silvestre \cite{Sil} we set $\Phi(x)= \frac{C_{n,\sigma}}{|x|^{n-2\sigma}}$ the fundamental solution of $(-\Delta)^{\sigma}$ and we construct $\Gamma$ from $\Phi$ by modifying $\Phi$ only in $B_1$ so that $\Gamma\in C^{\infty}(\R{n})$. Via a rescaling, we consider for $\lambda>0$ the function 
$$\Gamma_{\lambda}(x)=\frac{1}{\lambda^{n-2\sigma}}\Gamma\left(\frac{x}{\lambda}\right),$$
and also define $\gamma_{\lambda}(x):=(-\Delta)^{\sigma} \Gamma_{\lambda}(x)$. Notice that
\begin{equation}\label{gamma1}
\gamma_\lambda(x)=\frac{1}{\lambda^n}\gamma_1\left(\frac{x}{\lambda}\right).
\end{equation}
By \cite[Prop. 2.7]{Sil} $\gamma_\lambda\in C^\infty(\R{n})$. %For more properties of the functions $\Gamma_{\lambda}$ and $\gamma_{\lambda}$ we refer to \cite{Sil}.
We will need the following two results:
\begin{prop}[\cite{Sil}, Prop. 2.12]\label{2.2.3}
For $|x|>\lambda$, we have
\begin{equation}\label{gammalambda}
\gamma_{\lambda}(x)=\int_{B_{\lambda}(0)}\frac{\Phi(y)-\Gamma_{\lambda}(y)}{|x-y|^{n+2\sigma}}dy.
\end{equation}
%In particular, since $\Phi-\Gamma$ is compactly supported, $\gamma_\lambda(x)=O(|x|^{-n-2\sigma})$ as $$
\end{prop}
%The following can be seen as a replacement of the mean value property of harmonic functions.
\begin{prop}[\cite{Sil}, Prop. 2.22]\label{2.2.13}
Assume that $u\in L_{\sigma}(\R{n})$ such that $(-\Delta)^{\sigma}u=0$ in $\Omega\subset\R{n}$. Then $u\in C^0(\Omega)$ and $u(x)=u*\gamma_{\lambda}(x)$ for every $x\in \Omega$ and $\lambda\in (0,\mathrm{dist}(x,\de\Omega))$.
\end{prop}

We remark that, although our definition of $\Gamma$ (hence of $\Gamma_\lambda$ and $\gamma_\lambda$) is slightly different from the one in \cite{Sil}, the proofs of the above propositions go through with almost no change.

\begin{proof}[Proof of Proposition \ref{propschauder}] 
The proof uses Proposition \ref{2.2.13} and a standard convolution argument. For every $k\in\mathbb{N}\cup\{0\}$, we have from Proposition \ref{2.2.13} that $\nabla^k u=u*\nabla^k \gamma_\lambda$ (we use the notation that $\nabla^0$ is the identity operator) in $B_r$ for $\lambda= r/2$. Hence, for $x\in B_r$,
$$|\nabla^k u(x)|\leq \int_{\mathbb{R}^{n}\setminus B_{2r}}|u(y)| |\nabla^k\gamma_{\lambda}(x-y)|dy+\int_{B_{2r}}|u(y)| |\nabla^k\gamma_{\lambda}(x-y)|dy=:I+II.$$
Notice that
\[
\frac{1}{|x-y-z|^{n+2\sigma+k}}\le \frac{1}{(|y|-r-\lambda)^{n+2\sigma+k}}\le \frac{C_{n,\sigma,k}}{|y|^{n+2\sigma+k}}, \quad |y|>2r, |x|<r,|z|<\lambda= \frac{r}{2}.
\]
Then we have, by differentiating \eqref{gammalambda},
$$|\nabla^k\gamma_\lambda(x-y)|\le C_{n,\sigma,k}\int_{B_\lambda}\frac{|\Phi(z)-\Gamma_\lambda(z)|}{|x-y-z|^{n+2\sigma+k}}dz\le \frac{C_{n,\sigma,k} \lambda^{2\sigma}}{|y|^{n+2\sigma+k}},\quad |y|>2r, |x|<r,\lambda= \frac{r}{2}.$$
It follows that
$$I\le C_{n,\sigma} r^{2\sigma-k} \int_{\R{n}\setminus B_{2r}}\frac{|u(y)|}{|y|^{n+2\sigma}} dy.$$
As for $II$, notice that \eqref{gamma1} implies $\nabla^k \gamma_\lambda =\lambda^{-n-k}\nabla^k\gamma_1\big(\frac{x}{\lambda}\big)$, from which one bounds
$$II\le C_{n,\sigma,k}\|\nabla^k\gamma_{r/2}\|_{L^\infty}\int_{B_{2r}} |u(y)|dy \le \frac{C_{n,\sigma,k}}{r^{n+k}} \|u\|_{L^1(B_{2r})}.$$
The proof of \eqref{schaud2} is completed.
\end{proof}

\begin{cor}\label{corschauder} Suppose $u\in L_\sigma(\R{n})$ for some $\sigma\in (0,1)$ and $(-\Delta)^\sigma u=f$ in $B_2$ for some $f\in C^{k,\alpha}(B_2)$, where $\alpha\in (0,1), k\in \mathbb{N}\cup\{0\}$ and $\alpha+2\sigma$ is not an integer. Then $u\in C^{k,\alpha+2\sigma}(B_1)$ ($C^{k,\beta}(B_1)= C^{k+1, \beta-1}(B_1)$ if $\beta>1$). Moreover, 
\[
\|u\|_{C^{k,\alpha+2\sigma}(B_1)}\le C_{n,\sigma,k}\left( \int_{\R{n}}\frac{|u(x)|}{1+|x|^{n+2\sigma}}dx+\|f\|_{C^{k,\alpha}(B_2)}\right),
\]
where $C_{n,\sigma,k}$ is a positive constant depending only on $n,\sigma$ and $k$.
\end{cor}

\begin{proof} This can be proven similarly as in Proposition 2.8 of \cite{Sil}, by using the estimates in Proposition \ref{propschauder}.
\end{proof}

\noindent T. Jin

{\small
\noindent Department of Mathematics, The University of Chicago\\
5734 S. University Avenue, Chicago, IL 60637, USA\\
Email: \textsf{tj@math.uchicago.edu}
}
\medskip

\noindent A. Maalaoui

{\small
\noindent  Department of Mathematics and Computer Science, University of Basel\\
Rheinsprung 21 CH-4051, Basel, Switzerland \\
Email: \textsf{ali.maalaoui@unibas.ch}
}
\medskip

\noindent L. Martinazzi

{\small
\noindent  Department of Mathematics and Computer Science, University of Basel\\
Rheinsprung 21 CH-4051, Basel, Switzerland \\
Email: \textsf{luca.martinazzi@unibas.ch}
}

\medskip

\noindent J. Xiong

\small{
\noindent Beijing International Center for Mathematical Research, Peking University\\
Beijing 100871, China\\
Email: \textsf{jxiong@math.pku.edu.cn}
}

\begin{thebibliography}{99}
\small
\bibitem{AD} \textsc{Adimurthi, O. Druet,} \emph{Blow-up analysis in dimension $2$ and a sharp form of Trudinger-Moser inequality}, Comm. Partial Differential Equations \textbf{29} (2004), 295-322.

\bibitem{AS} \textsc{Adimurthi, M. Struwe,} \emph{Global compactness properties of semilinear elliptic equations with critical exponential growth}, J. Funct. Anal. \textbf{175} (2000), 125-167.

\bibitem{Beck} \textsc{W. Beckner}, \emph{Sharp Sobolev inequalities on the sphere and the Moser-Trudinger inequality}, Ann. of Math. (2) \textbf{138} (1993), 213-242.

\bibitem{Bra} \textsc{T. Branson}, \emph{Group representations arising from Lorentz conformal geometry}, J. Funct. Anal. \textbf{74} (1987), 199-293.

\bibitem{Bra2} \textsc{T. Branson}, \emph{Sharp inequality, the functional determinant and the complementary series}, Trans. Amer. Math. Soc. \textbf{347} (1995), 3671-3742.


\bibitem{cha} \textsc{S-Y. A. Chang}, 
Non-linear Elliptic Equations in Conformal Geometry, Zurich lecture notes in advanced mathematics, EMS (2004).


\bibitem{CC} \textsc{S.-Y. A. Chang, W. Chen,}
\emph{A note on a class of higher order conformally covariant equations},
Discrete Contin. Dynam. Systems \textbf{7} (2001), no. 2, 275--281.

\bibitem{CQ} \textsc{S-Y. A. Chang, J. Qing}, \emph{The Zeta functional determinants on manifolds with boundary. I. The formula.} J. Funct. Anal. \textbf{147} (1997), 327–362.

\bibitem{CY2} \textsc{S-Y. A. Chang, P. Yang}, 
\emph{On uniqueness of solutions of $n$-th order differential equations in conformal geometry}, Math. Res. Lett. \textbf{4} (1997), 91-102.

\bibitem{CL} \textsc{W. Chen, C. Li}, \emph{Classification of solutions of some nonlinear elliptic equations}, Duke Math. J. \textbf{63} (3) (1991), 615-622.

\bibitem{dru} \textsc{O. Druet}, \emph{Multibumps analysis in dimension $2$: quantification of blow-up levels}, Duke Math. J. \textbf{132} (2006), 217-269.

\bibitem{DR} \textsc{O. Druet, F. Robert,} \emph{Bubbling phenomena for fourth-order four-dimensional PDEs with exponential growth}, Proc. Amer. Math. Soc \textbf{3} (2006), 897-908.

\bibitem{GZ} \textsc{C. R. Graham, M. Zworski}, \emph{Scattering matrix in conformal geometry}, Invent. Math. \textbf{152} (2003), 89-118.

\bibitem{HM} \textsc{A. Hyder, L. Martinazzi}, \emph{Conformal metrics on $\R{2m}$ with
constant Q-curvature,  prescribed volume and asymptotic behavior}, preprint (2014).

\bibitem{LRS} \textsc{T. Lamm, F. Robert, M. Struwe}, \emph{The heat flow with a critical exponential nonlinearity}, J. Funct. Anal. \textbf{257} (2009) 2951-2998.

\bibitem{LL} \textsc{E.H. Lieb, M. Loss}, Analysis. 
Second edition. Graduate Studies in Mathematics, 14. American Mathematical Society, Providence, RI, 2001. ISBN: 0-8218-2783-9. 

\bibitem{lin1}\textsc{C-S. Lin,}
\emph{A classification of solutions of a conformally invariant fourth order equation in $\R{n}$},
Comment. Math. Helv. \textbf{73} (1998), no. 2, 206-231. 


\bibitem{mal} \textsc{A. Malchiodi}, \emph{Compactness of solutions to some geometric fourth-order equations}, J. reine angew. Math. \textbf{594} (2006), 137-174.

\bibitem{MS} \textsc{A. Malchiodi, M. Struwe,} \emph{$Q$-curvature flow on $S^4$}, J. Diff. Geom. \textbf{73} (2006), 1-44.

\bibitem{mar1} \textsc{L. Martinazzi,}  \emph{Classification of solutions to the higher order Liouville's equation on $\mathbb{R}^{2m}$}, Math. Z. \textbf{263} (2009), 307--329.

\bibitem{mar3} \textsc{L. Martinazzi,} \emph{Concentration-compactness phenomena in higher order Liouville's equation}, J. Funct. Anal. \textbf{256} (2009), 3743-3771

\bibitem{mar4} \textsc{L. Martinazzi}, \emph{A threshold phenomenon for embeddings of $H^m_0$ into Orlicz
spaces}, Calc. Var. Partial Differential Equations \textbf{36} (2009), 493-506.

\bibitem{mar5} \textsc{L. Martinazzi}, \emph{Quantization for the prescribed Q-curvature equation on open domains}, Commun. Contemp. Math. \textbf{13} (2011), 533-551.

\bibitem{mar2} \textsc{L. Martinazzi,}  \emph{Conformal metrics on $\R{2m}$ with constant $Q$-curvature and large volume}, to appear in Ann. Inst. Henri Poincar\'e (C) (2013).

\bibitem{MS2} \textsc{L. Martinazzi, M. Struwe},  \emph{Quantization for an elliptic equation of order 2m with critical exponential non-linearity }, Math Z. \textbf{270} (2012), 453-487.

\bibitem{MP} \textsc{L. Martinazzi, M. Petrache,} \emph{Asymptotics and quantization for a mean-field equation of higher order}, Comm. Partial Differential Equation \textbf{35} (2010), 1-22.

\bibitem{Sma} \textsc{R. S. Palais}, \emph{The principle of symmetric criticality}, Comm. Math. Phys. \textbf{69} (1979), 19-30.

\bibitem{rob1} \textsc{F. Robert}, \emph{Concentration phenomena for a fourth order equation with exponential growth: the radial case}, J. Differential Equations \textbf{231} (2006), 135-164.

\bibitem{rob2} \textsc{F. Robert}, \emph{Quantization effects for a fourth order equation of exponential growth in dimension four}, Proc. Roy. Soc. Edinburgh Sec. A \textbf{137} (2007), 531-553.

\bibitem{RS} \textsc{F. Robert, M. Struwe}, \emph{Asymptotic profile for a fourth order PDE with critical exponential growth in dimension four}, Adv. Nonlin. Stud. \textbf{4} (2004), 397-415.

\bibitem{RW} \textsc{F. Robert, J. Wei,} \emph{Asymptotic behavior of a fourth order mean field equation with Dirichlet boundary condition}, Indiana Univ. Math. J. \textbf{57} (2008), 2039-2060.


\bibitem{Sil} \textsc{L. Silvestre}, \emph{Regularity of the obstacle problem for a fractional power of the Laplace operator}. Comm. Pure Appl. Math. \textbf{60} (2007), no. 1, 67-112.

\bibitem{Stein} \textsc{E. M. Stein,} Singular integrals and differentiability properties of functions.
Princeton Mathematical Series, No. 30 Princeton University Press, Princeton, N.J. 1970.

\bibitem{str0} \textsc{M. Struwe,} \emph{A flow approach to Nirenberg's problem}, Duke Math. J. \textbf{128}(1) (2005), 19-64.

\bibitem{str} \textsc{M. Struwe,} \emph{Quantization for a fourth order equation with critical exponential growth}, Math. Z. \textbf{256} (2007), 397-424.


\bibitem{WX1} \textsc{J. Wei, X. Xu,} \emph{Classification of solutions of higher order conformally invariant equations}, Math. Ann \textbf{313} (1999), 207-228.

\bibitem{WX} \textsc{J. Wei, X. Xu,} \emph{Prescribing $Q$-curvature problem on $\mathbb{S}^n$},
 J. Funct. Anal. \textbf{257} (2009), no. 7, 1995--2023.

\bibitem{W-Y} 
\textsc{J. Wei, D. Ye,}
\emph{Nonradial solutions for a conformally invariant 
fourth order equation in $\mathbb{R}^4$.} Calc. Var. Partial Differential Equations \textbf{32} (2008), no. 3, 373-386.


\bibitem{Xu}\textsc{X. Xu,}
\emph{Uniqueness and non-existence theorems for conformally invariant equations},
J. Funct. Anal. \textbf{222} (2005), no. 1, 1--28. 
\end{thebibliography}
\end{document}